\documentclass[12pt]{amsart}

\usepackage[applemac]{inputenc}

\usepackage{dsfont,color,fullpage}
\usepackage{bezier,epic}

\usepackage[german,english]{babel}
\newtheorem{thm}{Theorem}
\newtheorem{prop}[thm]{Proposition}
\newtheorem{lemma}[thm]{Lemma}
\newtheorem{cor}[thm]{Corollary}

\newcommand{\tto}{\longrightarrow}

\newcommand{\CC}{\mathbb C}

\newcommand{\bfami}{\mathcal B}

\newcommand{\dfami}{\mathcal D}

\newcommand{\ord}{\operatorname{ord}}

\newcommand{\inv}{^{^{-1}}}

\begin{document}

\title{On a discriminant knot group problem of Brieskorn}

\author{Michael L\"onne}
\address{Michael L\"onne\\ Mathematik VIII\\ Universit\"at Bayreuth\\
Universit\"atstrasse 30\\  \phantom{xxxx} \hspace*{1cm}95447 Bayreuth\\ Germany
}
\email{michael.loenne@uni-bayreuth.de}

\begin{abstract}
Quite some time ago, at the singularity conference at Carg\`ese 1972
Brieskorn asked the following question:
\begin{quote}
Is the local fundamental group 
$\pi_1^{s}(S-D)$ of the discriminant complement inside the
semi-universal unfolding $S$ of an isolated hypersurface singularity
constant for $s$ in the $\mu$-constant stratum $\Sigma_E$?
\end{quote}
We review this question and give an affirmative answer in case of
singular plane curve germs of multiplicity at most $3$.
\end{abstract}
\thanks{The present work took place in the framework  of the 
 ERC Advanced grant n. 340258, `TADMICAMT' }

\maketitle

\section{Introduction}

The question of Brieskorn was published in Astérisque 7-8, \emph{Colloque
sur les singularités en géometrie analytique}. 
In that article Brieskorn gives a summary of the problems and questions he
considers central in the investigation of monodromy, and their answers which
-- as he writes -- will help much to arrive at a more profound understanding, 
\cite{BrieskornCargese}.

In Brieskorns view the local fundamental group of the discriminant
complement -- the discriminant knot group as it will be called in
the present article  -- 
lies at the heart of the study of the algebraic monodromy and
the intersection lattice of the Milnor fibre and should soon reveal
to contain more or less the same amount of information.

This optimism probably resulted in the spectacular success in the
study of simple hypersurface singularities where Brieskorn himself 
made important contributions, \cite{MR0293615,MR0437798,MR0323910}. For the simple singularities the algebra, the
geometry and the combinatorial group theory are most closely tied
together and hope was widespread to get similar results for
more general singularities under suitable forms of relaxation.

However, the topology of the discriminant complement remains
a mystery to the present day, and only little progress has been
made on the problems Brieskorn addressed to it.

In this article we will review the problem stated in the abstract
\begin{quote}
Is the discriminant knot group $\pi_1^{s}(S-D)$ of an isolated 
hypersurface singularity constant for $s$ in the $\mu$-constant 
stratum?
\end{quote}

At the time of writing the evidence in favour of a positive answer had two aspects.
First in the case of simple singularities the answer is trivially positive.
Second, the homomorphic image under algebraic monodromy is 
constant along the $\mu$-constant stratum.

On the other hand an article of Pham \cite{PhamCargese}
presented at the very conference at Carg\`ese
was interpreted by Brieskorn as evidence in favour of a negative answer: Pham
showed that the topological type of the generic discriminant curve of certain plane curve singularities of multiplicity $m=3$ is not constant along the $\mu$-constant
stratum.

In fact, Brieskorn proposes to study the discriminant knot group by the local
Zariski hyperplane theorem as proved by Le and Hamm \cite{hl}:
\[
\pi_1(S-D) \quad \cong\quad
\pi_1(H- H\cap D)
\]
where $H$ is a plane in $S$ parallel to
a generic plane $H_0\neq H$ through the origin.
$H_0\cap D$ is called \emph{a generic discriminant curve} and 
$H\cap D$ a corresponding \emph{unfolded generic discriminant curve}.
The topological type of the former is constant along the $\mu$-constant stratum 
if and only if the topological type of the latter is.

Therefore the result of Pham shows that the line of argument which Brieskorn
had in mind cannot work.

In this article, however, we will follow Brieskorns strategy and bridge the gap
by using a stronger form of the Zariski van Kampen method applicable to
more general plane sections of the discriminant.

We will turn the Pham examples into evidence for a positive answer to
Brieskorns problem by the following theorem.

\begin{thm}
Suppose $f$ is a plane curve singularity of multiplicity at most $3$,
then the discriminant knot group is constant along the $\mu$-constant
stratum.
\end{thm}

As remarked before, in the case of simple singularities the
claim trivially holds true. 
By classification this settles the case of multiplicity $2$ and
of plane curve singularities of Milnor number at most $8$.

As a direct corollary we can sharpen the result of \cite{IJM}. Suppose $f$
is topologically equivalent to a plane curve singularities of Brieskorn Pham type
of multiplicity $3$
\[
f \quad \sim_{top} \quad y^3 + x^{\nu+1} \qquad \mbox{for some } \nu\geq 2,
\]
then $f$ is a $\mu$-constant deformation of the Brieskorn Pham singularity 
and has the same distinguished Dynkin diagram

\begin{figure}[ht]
\setlength{\unitlength}{1.8mm}
\begin{picture}(60,15)(-7,-4)

\drawline(2,0)(8,0)
\drawline(6,0)(4,-1)
\drawline(6,0)(4,1)
\drawline(2,10)(8,10)
\drawline(6,10)(4,9)
\drawline(6,10)(4,11)
\drawline(0,2)(0,8)
\drawline(0,6)(-1,4)
\drawline(0,6)(1,4)
\drawline(12,0)(18,0)
\drawline(16,0)(14,-1)
\drawline(16,0)(14,1)
\drawline(12,10)(18,10)
\drawline(16,10)(14,9)
\drawline(16,10)(14,11)
\drawline(20,2)(20,8)
\drawline(20,6)(19,4)
\drawline(20,6)(21,4)
\drawline(10,2)(10,8)
\drawline(10,6)(9,4)
\drawline(10,6)(11,4)
\drawline(40,2)(40,8)
\drawline(40,6)(39,4)
\drawline(40,6)(41,4)

\drawline(1.5,1.5)(8.5,8.5)
\drawline(6,6)(5,3.6)
\drawline(6,6)(3.6,5)
\drawline(11.5,1.5)(18.5,8.5)
\drawline(16,6)(13.6,5)
\drawline(16,6)(15,3.6)

\bezier{2}(27,0)(30,0)(33,0)
\bezier{2}(27,10)(30,10)(33,10)

\put(0,0){\circle*{1}}
\put(10,0){\circle*{1}}
\put(10,10){\circle*{1}}
\put(0,10){\circle*{1}}
\put(20,0){\circle*{1}}
\put(20,10){\circle*{1}}
\put(40,0){\circle*{1}}
\put(40,10){\circle*{1}}

\put(1,-3){11}
\put(11,-3){12}
\put(41,-3){$1\nu$}
\put(11,7){22}
\put(41,7){$2\nu$}
\put(1,7){21}
\put(21,-3){13}
\put(21,7){23}

\end{picture}
\caption{Dynkin diagram of $y^3+x^{\nu+1}$}
\label{dyn}
\end{figure}
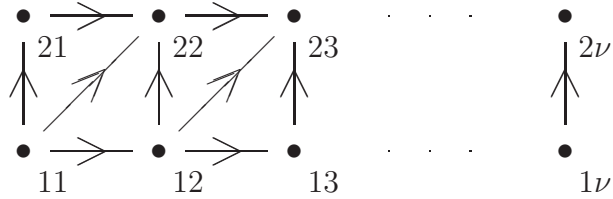

\noindent
where the set $V$ of vertices is ordered by the lexicographic order of their double 
indices and the set $E$ of oriented edges contains the pair of corresponding
vertices only in their proper order.

Since the discriminant knot group by Theorem $1$ is the same for $f$ and the 
Brieskorn Pham singularity we get from \cite[Thm1.1]{IJM}.
(Does it extend to all cases in \cite{CR07}?)

\begin{thm}
Suppose $f$ is topologically equivalent to a Brieskorn Pham polynomial
$y^3+x^{\nu+1}$ then its discriminant knot group is presented by
\[
\left\langle t_i,\: i\in V \:\Bigg | \:
\begin{array}{cl}
t_i t_j = t_j t_i & (i,j),(j,i) \not\in E \\
t_i t_j t_i = t_j t_i t_j & (i,j) \in E \\
t_i t_k t_j t_i = t_j t_i t_k t_j & (i,j),(i,k),(j,k) \in E
\end{array}
\right\rangle
\]
\end{thm}

A step beyond the result of this article might address the case
of unimodal hypersurface singularities. Possibly it is
sufficient to look at the generic discriminant curve, since in the
cases not covered by our result, Greuel \cite{Gr77,Gr78} 
has shown that at least
the number of cusps of the unfolded generic discriminant 
curve is constant along the $\mu$-constant stratum.

\section{Review of the results of Pham}

In his article \cite{PhamCargese} Pham provides a careful analysis of the generic 
discriminant curve in case of a plane curve singularity of multiplicity $3$
\[
f \quad = \quad y^3 - P(x) y + Q(x).
\]
While skipping his calculation which we will mimic in the next section,
here we only want to introduce the minimum of notation to state his results
and draw some first conclusions towards the proof of our main theorem.

In addition to the notorious Milnor number
\[
\mu \quad = \quad \dim \CC[X,Y] / \langle f_x, f_y \rangle
\]
Pham needs the analytic \emph{$\sigma$-invariant} associated to the ideal
generated by $f$ and its derivatives up to second order
\[
\sigma \quad = \quad1+ \dim \CC[X,Y] / \langle f, f_x, f_y, f_{xx}, f_{xy}, f_{yy} \rangle.
\]
He also 
gives some useful formulas for calculations:

\begin{lemma}[\cite{PhamCargese} \S1,p.366]
\label{sigma}
If $f$ is a function germ as above, the analytic invariant $\sigma$
is given by
\[
\sigma \quad=\quad \min \{ \ord P, \ord Q' \}
\]
and the Milnor number is given by
\[
\mu \quad=\quad \ord (3 Q'^2 - P P'^2).
\]
\end{lemma}


Instead of citing the main result in its full strength, which is a complete topological
classification of generic discriminant curves, we distill the essence,
what we will need below.

\begin{prop}[cf. \cite{PhamCargese}]
\label{pham}
The topological type of the generic discriminant curve
only depends on the topological invariant $\mu$ and the analytic
invariant $\sigma$.
\end{prop}

\begin{cor}
\label{pham_cor}
The topological type of the unfolded generic discriminant curve
only depends on the topological invariant $\mu$ and the analytic
invariant $\sigma$.
\end{cor}

\begin{proof}
The topological type of the generic discriminant curve determines its
Milnor number $\tilde \mu$.
The number $\mu+\tilde\mu-1$ is the sum
of three times the number of cusps and twice the number of nodes 
of any corresponding unfolded discriminant curve, cf. \cite{PhamCargese}.
Since both cardinalities are upper semi-continuous
and the set with constant $\sigma$ and $\mu$ is connected, they are both
constant along this set,
and so is the topological type of the unfolded generic discriminant
curve.
\end{proof}

\begin{prop}
\label{nonBP}
If $f$ is a plane curve singularity of multiplicity $3$ and 
\[
f \quad \not\sim_{top} \quad y^3 + x^{\nu+1} \qquad\mbox{for all } \nu
\]
then the discriminant knot group is constant along the $\mu$-constant
stratum.
\end{prop}

\begin{proof}
According to the classification by Arnol'd \cite{ArnoldInvent} $f$ is simple, 
of type $J_{k,i}$, $k\geq2,i>0$, or of type $E_{6k+1}$, $k\geq2$.
In the simple case the claim is trivially true as was remarked before. 

In case of $f\in J_{k,i}$, $k\geq2,i>0$ Arnol'd has given a normal form which
by an analytic equivalence -- more precisely  
by a Tschirnhaus transformation -- can be put in the form considered by Pham.
\begin{eqnarray*}
 & & y^3 + y^2x^k + a(x) x^{3k+i},\quad \ord a = 0 \\
 & \sim_{an} & y^3 -\frac13 yx^{2k} + \frac2{27} x^{3k} + a(x) x^{3k+i}
\end{eqnarray*}
According to the lemma $\sigma=2k$ and thus $\sigma$ is independent of $a(x)$.

In case of $f\in E_{6k+1}$ the normal form of Arnol'd is in the form
considered by Pham, so from
\[ 
y^3 +yx^{2k+1}+ a(x) x^{3k+2}
\]
$\sigma=2k+1$ independent of $a(x)$ is immediate by the lemma again.

In both cases we conclude with the corollary that the topological type
of the unfolded generic discriminant is constant along the $\mu$-constant
stratum.
Therefore the fundamental groups of their complements also do not change.
The local Zariski theorem on hyperplane sections \cite{hl} identifies 
these groups with the discriminant knot groups which are thus shown to be 
constant along the $\mu$-constant stratum. 
\end{proof}

\section{Existence of suitable non-generic discriminant curves}

In this section we follow the path traced by Pham to obtain a non-generic
reduced discriminant curve which does not change its topological type
under a small deformation along the $\mu$-constant stratum,
although the analytic invariant $\sigma$ changes.

In fact, as the last section will prove, it will suffice to do so for the
Brieskorn Pham polynomials.
\medskip

We recall from \cite{PhamCargese} the construction of the discriminant 
curve in direction of a linear perturbation by a polynomial $p(x)y+q(x)$.
The critical set of the unfolding of 
\[
f\quad=\quad f(x,y)\quad = \quad y^3 - P_0(x)y + Q_0(x)
\]
by $-u + t (p(x)y+q(x))$ is
a curve in $4$-space and the corresponding discriminant curve
is obtained by projection along the coordinates $x,y$,
algebraically by elimination of $x,y$ from
\begin{eqnarray}
\label{function}
u & = & \:y^3 - P y + Q \\
\label{delfx}
0 & = & \quad\:\:- P'y + Q' 
\\ \label{delfy}
0& = & 3y^2 - P 
\end{eqnarray}
where $P=P_0+tp$, $Q=Q_0+tq$.

But as Pham does, we take the detour by the projection along $u$ and $y$
which is easier to obtain.
The parameter $u$ is eliminated by the sole use of \eqref{function}
and from  \eqref{delfx} and \eqref{delfy} we can eliminate $y$ to get
\begin{equation}
\label{xtplane}
3 Q'^2 -PP'^2 \quad =\quad 0
\end{equation}
First we consider for an additional parameter $s=0$ or sufficiently small the case
\[
Q \quad = \quad Q_0 \quad=\quad x^{\nu+1},
\qquad q=0,
\qquad
P \quad = \quad P_0 + tp \quad = \quad sx^\sigma +tx.
\]

The first step according to Pham is to compute the branches $x(t)$.
Recall the expansion of \eqref{xtplane} in terms of the variable $t$ according to
\[
3 (Q'_0+tq')^2 -(P_0+tp)(P'_0+tp')^2 \quad=\quad 
A_0 + A_1 t + A_2 t^2 + A_3 t^3
\]
In the current situation we get the following vanishing orders of the $A_i$
under the assumption of $s$ sufficiently small.

\begin{center}
\begin{tabular}{lclclcc}
&&expansion & & &$\quad$& $\quad$vanishing order 
\\[3mm]
$A_0$ & = & $3 Q'^2_0 -P_0P'^2_0$
& = & $3(\nu+1)^2x^{2\nu}-s^3\sigma^2x^{3\sigma-2}$
&& $\mu = \min\{2\nu, 3\sigma - 2\} $ 
\\[2mm]
$A_1$ &=& $-xP'^2_0 - 2P_0P'_0$
& = & $-s^2\sigma(\sigma+2)x^{2\sigma-1}$
&& $2\sigma -1$ for $s\neq0$
\\[2mm]
$A_2$ &=& $- P_0 -2 x P'_0$
& = & $-s(2\sigma+1)x^\sigma$
&& $\sigma $ for $s\neq0$
\\[2mm]
$A_3$ &=& $-x $
&&& & $1$
\end{tabular}
\end{center}
\medskip


Under the assumption $3\sigma-2\geq 2\nu$ the Newton Polygon looks as below
depending on whether equality holds or not.
( The $\circ$ are only present for $s\neq0$. )
\medskip

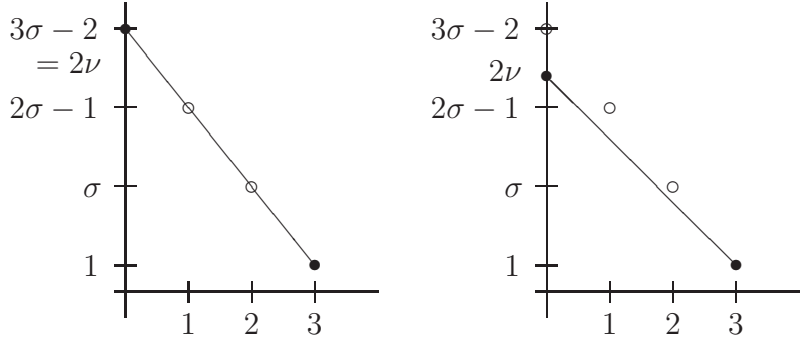
\setlength{\unitlength}{7mm}
\begin{picture}(14,6)
\put(2,0){
\begin{picture}(10,6)
\put(2.2,0){\line(0,1){6}}
\put(2,.5){\line(1,0){5}}

\put(2,1){\line(1,0){.4}}
\put(1.4,.8){$1$}
\put(2,2.5){\line(1,0){.4}}
\put(1.4,2.3){$\sigma$}
\put(2,4){\line(1,0){.4}}
\put(0,3.8){$2\sigma-1$}
\put(2,5.5){\line(1,0){.4}}
\put(0,5.3){$3\sigma-2$}
\put(0.5,4.65){$= 2\nu$}

\put(3.4,.3){\line(0,1){.4}}
\put(3.25,-.3){$1$}
\put(4.6,.3){\line(0,1){.4}}
\put(4.45,-.3){$2$}
\put(5.8,.3){\line(0,1){.4}}
\put(5.65,-.3){$3$}

\put(2.2,5.5){\line(4,-5){3.6}}
\put(2.2,5.5){\circle*{.2}}
\put(3.4,4){\circle{.2}}
\put(4.6,2.5){\circle{.2}}
\put(5.8,1){\circle*{.2}}

\end{picture}
}

\put(10,0){
\begin{picture}(10,6)
\put(2.2,0){\line(0,1){6}}
\put(2,.5){\line(1,0){5}}

\put(2,1){\line(1,0){.4}}
\put(1.4,.8){$1$}
\put(2,2.5){\line(1,0){.4}}
\put(1.4,2.3){$\sigma$}
\put(2,4){\line(1,0){.4}}
\put(0,3.8){$2\sigma-1$}
\put(2,5.5){\line(1,0){.4}}
\put(0,5.3){$3\sigma-2$}
\put(1.1,4.5){$2\nu$}

\put(3.4,.3){\line(0,1){.4}}
\put(3.25,-.3){$1$}
\put(4.6,.3){\line(0,1){.4}}
\put(4.45,-.3){$2$}
\put(5.8,.3){\line(0,1){.4}}
\put(5.65,-.3){$3$}

\put(5.8,1){\line(-1,1){3.6}}
\put(2.2,5.5){\circle{.2}}
\put(3.4,4){\circle{.2}}
\put(4.6,2.5){\circle{.2}}
\put(5.8,1){\circle*{.2}}
\put(2.2,4.6){\circle*{.2}}

\end{picture}
}

\end{picture}
\medskip

The leading term corresponding to the compact face has no multiple root.
This is obvious in case of $3\sigma-2> 2\nu$ and for $s=0$, therefore
it is true also for $s$ sufficiently small.

In particular, for $s$ sufficiently small, the number of branches is
constant and the leading term of each branch has non-vanishing
coefficient which varies continuously with $s$.
\medskip

We consider now the case $s=0$ in detail ( but claims hold true also
for $s$ small up to continuous changes of the coefficients ) and distinguish
the following cases
\begin{enumerate}
\item[$(a)$] $\gcd(2\nu-1,3) = \gcd(2\nu+2,3) = \gcd(\nu+1,3)=1$
\item[$(b)$] $\gcd(2\nu-1,3) =  \gcd(\nu+1,3)=3$, \\
$\gcd(3\nu+3,2\nu-1) = \gcd(6\nu+6,2\nu-1) = \gcd(2\nu-1,9) = 3$ 
\item[$(c)$] $\gcd(2\nu-1,3) =  \gcd(\nu+1,3)=3$, 
$\gcd(3\nu+3,2\nu-1) = \gcd(2\nu-1,9) = 9$ 
\end{enumerate}
In case $(a)$ there are two branches
\begin{equation}
\label{abranches}
x(t) = 0, \quad x(t) = c_0 t^\frac3{2\nu-1} + h.o.t.
\end{equation}
in cases $(b)$ and $(c)$ there are four branches
\begin{equation}
\label{bcbranches}
x(t) = 0, \quad 
x(t) = c_0 \omega^i t^\frac3{2\nu-1} + h.o.t.,\quad
 i=0,1,2.
\end{equation}
where $c_0\neq 0$ is a numerical constant and $\omega$ a primitive
root of unity of order $2\nu-1$.

To pursue along the lines of \cite{PhamCargese} we check 
first that the hypothesis
\[
P'(x(t),t) \quad = \quad P'_0(x(t)) + tp'(x(t))
\quad = \quad P'_0(x(t)) + t
\quad \neq \quad 0
\quad \in\quad \CC\{ t^{\frac1{2\nu-1}}\}
\]
holds true for every possible branch $x(t)$.

Therefore the following formula derived by Pham is valid in the field
of fractions $\CC(( t^{\frac1{2\nu-1}}))$.
\begin{eqnarray}
\label{uformel}
u & = & -\frac23 \frac{P}{P'} Q' + Q \\
& = & \big(-\frac23 (\nu+1) + 1\big)\: x(t)^{\nu+1} \notag
\end{eqnarray}

In case $(a)$ we plug in the expansions \eqref{abranches} to get
\[
u(t) = 0, \quad
u(t) = \big(-\frac23 (\nu+1) + 1\big)c_0^{\nu+1}t^\frac{3\nu+3}{2\nu-1}
+ h.o.t.
\]
The corresponding branches are reduced and not equal. Moreover
the second expansion does not have further essential summands, 
since the exponent of $t$ is in its reduced form and has the maximal
possible denominator.

In case $(b)$ we write $2\nu-1=3e$ with $e$ coprime to $3$
and get the expansions
\[
u(t) = 0, \quad
u(t) = \big(-\frac23 (\nu+1) + 1\big)c_0^{\nu+1} \omega^{(\nu+1)i}
t^\frac{3\nu+3}{2\nu-1}
+ h.o.t., \quad  i = 0,1,2
\]
Again the corresponding branches are reduced and pairwise not
equal.
This time the reduced form of the exponent has denominator $e$.
Again this is the maximal possible denominator, since the $u$-degree
of the Weierstrass polynomial of the first branch is $1$ and of the 
other three branches is the maximal denominator, but their sum
is equal to the Milnor number which is $\mu=3e + 1$.

Thus in case $(a)$ and case $(b)$ we have found a perturbation
such that the topological type of corresponding discriminant curve
does not vary for $s$ sufficiently small.

In case $(c)$ we write $\nu-5=9\rho$, but we fail to argue
as above. In fact, for $s=0$ we get expansions which parametrize
the branches of the corresponding discriminant curve by a $3:1$ map
so this curve is non-reduced.

Hence we rerun the method of Pham with the modified perturbation
\[
t( xy + x^{3\rho+4}),\quad i.e.\quad p = x,\: q = x^{3\rho+4}
\]
The essential expansion of $x$ in terms of $t$ remains
the same as before, since the new perturbation only adds
the points $(1, 12\rho+8)$ and $(2,6\rho+6)$
to the support, which both lie above the Newton polygon.
\[
x(t) = 0, \quad 
x(t) = c_0 \omega^i t^\frac3{2\nu-1} + h.o.t.,\quad
 i=0,1,2
\]
The reduced form of the exponent is the inverse of $6\rho+3$.

The formula \eqref{uformel} now gives (using $c_\nu,c_\rho$
for the obvious constants)
\begin{eqnarray*}
u & = & 
\big(-\frac23 (\nu+1) + 1\big)\: x(t)^{\nu+1}  +
\big(-\frac23 (3\rho+4) + 1\big)\: tx(t)^{3\rho+4} 
\\
& = & 
c_\nu c_0^{\nu+1}\omega^{(\nu+1)i}t^\frac{\nu+1}{6\rho+3}  +
c_\rho c_0^{3\rho+4}\omega^{(3\rho+4)i}t^\frac{\nu+2}{6\rho+3}
+ h.o.t.
\end{eqnarray*}
We can now argue as before, that $6\rho+3$ is the maximal
possible denominator. Therefore no further essential
summand occurs, and we get reduced, pairwise distinct branches
also in the remaining case $(c)$.

Let us summarize the results of the present section as follows.

\begin{prop}
\label{exist}
Suppose $f=y^3 + x^{\nu+1}$ and $m$ is
an integer with 
\[
2\nu\leq 3m -2 , \quad m\leq \nu
\]
then there exists a $3$-parameter unfolding $F(x,y;u,t,s)$,
such that
\begin{enumerate}
\item
along $u=t=0$ the unfolding is $\mu$-constant
\item
for fix $s$ sufficiently small, the discriminant curve 
of the unfolding $F_{s}$ by the parameter $t$ is 
reduced and topologically equivalent to that of $F_{0}$.
\item
the analytic invariant $\sigma$ is $\nu$ for $s=0$ and
$m$ for $s\neq0$ sufficiently small.
\end{enumerate}
\end{prop}

\section{The Zariski theorem}

In this final section we have to revisit the local Zariski and van
Kampen theorem which avoids the use of generic hyperplane
sections, cf.\ the more extended exposition in \cite{IJM,L11}.

%

Our main interest lies in the discriminant
complement, so let us recall the basic setting: Given a holomorphic
function germ $f=f(x,y)$ on the germ $\CC^2,0$ of the affine plane with
coordinates $x,y$, we consider a versal unfolding, which can be given by
a function germ on the affine space germ $\CC^2,0\times\CC,0\times\CC^k,0$
\[
F(x,y,u,v) \quad = \quad f(x,y) - u + \sum_i^k v_i g_i(x,y)
\]
where the $g_i$ generate additively the local ideal of function germs
on $\CC^2,0$ vanishing at the origin up to elements in the Jacobian 
ideal of $f$. They are typically taken to be non-constant monomials.

We get a diagram
$$
\begin{array}{rcccccl}
(u,v_1,...,v_{k}) & \in & \CC^{k+1},0 & \supset & \dfami & = & \{(u,v)
\,|\, F\inv_{u,v}(0)\text{ is singular }
\}\\[2mm]
\big\downarrow\hspace*{10mm} & & \raisebox{1mm}{$p$}\big\downarrow \phantom{pp} && 
\\[2mm]
(v_1,...,v_{k}) & \in & \CC^{k},0 & \supset & \bfami & = & \{u
\,|\,F_{0,v} \text{ is not Morse }\}
\end{array}
$$
The restriction $p$\raisebox{-.7mm}{$|_\dfami$}
of the projection to the \emph{discriminant} $\dfami$ is
a finite map, such that the branch set coincides
with the \emph{bifurcation set} $\bfami$
and the critical points are contained in the pre-image $\tilde\bfami=p\inv(\bfami)$.
In particular, the origin is an isolated point in the intersection of
$\dfami$ with the fibre $p^{-1}(0)$.
If a hypersurface germ has this property we call it \emph{horizontal}
for the projection $p$.

The key observation 
is, that a suitable representative of the complement
of $\tilde\bfami$ is a trivial disc bundle by $p$ into which
$\dfami$ is embedded as a smooth submanifold, which is 
a connected topological cover by $p$.
This situation, which can be treated also in the language of
polynomial covers, cf.\ \cite{hansen},
naturally gives rise to a \emph{braid monodromy homomorphism}:
The domain is the fundamental group of the complement of
$\bfami$, its target is the group of mapping classes of the punctured fibre,
the image is called the \emph{braid monodromy group}.

It coincides with the map of fundamental groups
induced by the map of Lyashko Looijenga under
the natural identification of the mapping class group with the
fundamental group of the space of monic simple
univariate polynomials of degree $n$ at the corresponding base point:
\begin{eqnarray*}
\CC^{k} - \bfami & \tto & \CC[x],
\\
v & \mapsto & p_v
\end{eqnarray*}
which maps to monic univariate polynomials 
of degree $\mu$ with simple roots only and at the 
corresponding points of $\dfami$.


To use the braid monodromy group on the fundamental group
of the dis\-cri\-mi\-nant complement we employ the argument 
of Zariski and van Kampen \cite{vK}. It relies on a choice
of a geometric basis in the fibre over the base point which is the
customary tool to identify the action of the group of isotopy classes
of diffeomorphisms on the fundamental group of the fibre with
the right Artin action of the abstract braid group on the free generators
$t_1,\dots,t_n$ given by 
\[
(t_j)\sigma_j=t_jt_{j+1}t_j^{-1},\quad
(t_{j+1})\sigma_j = t_j,\quad
(t_i)\sigma_j= t_i, \mbox{if } i\neq j, j+1
\]

\begin{thm}[van Kampen]
Given a horizontal hypersurface germ with braid monodromy group
generated by braids $\{\beta_s\}$ in $Br_n$.
Then the local fundamental group of the complement is
finitely presented as
\[
\pi_1 =
\langle t_1,\dots,t_n\,|\, t_i^{-1} t_i^{\beta_s}, 1\leq i\leq n, \mbox{all }
\beta_s \rangle
\]
\end{thm}

The consideration above applies again to the hypersurface germ
$\bfami$ in the affine space germ $\CC^k,0$ provided we find
a projection for which $\bfami$ is horizontal.
In fact this put a constraint on a discriminant curve as
we will see in the following proof.



\begin{prop}
Let $g_1$ be a bivariate polynomial germ vanishing at $0$
such that the discriminant curve of the unfolding
\[
f - u + t g_1
\]
is reduced.
Then the fundamental group of the complement of a corresponding
unfolded discriminant curve is equal to the discriminant knot
group of $f$.
\end{prop}

\begin{proof}
Without loss of generality we may assume that $g_1$ is the first of the
functions in the versal unfolding of $f$ we consider.
Hence the complement of the unfolded discriminant curve
is a vertical plane section of the discriminant.
(At this point we could conclude with the local Zariski hyperplane
section theorem, if this vertical plane were known to be generic.)

By the van Kampen theorem, it suffices to show
that the two braid monodromy groups are equal.
They in turn are homomorphic images of the corresponding
fundamental groups of complements to the bifurcation set.
\medskip

If the discriminant curve is reduced, then the corresponding curve 
in the affine space germ $\CC^k,0$ does not belong to the bifurcation set,
otherwise, the discriminant curve has less than $\mu$
points over every $t$ and is non-reduced.

We deduce that the bifurcation set is horizontal for the projection
along the coordinate corresponding to $g_1$, since the $0$-fibre of that
projection was just shown not to be in the bifurcation set.

In particular the fundamental group in a generic vertical line
is generated by elements corresponding to a geometric basis.
They also generate the fundamental group of the complement to
the bifurcation set by the van Kampen theorem.

Put differently the fundamental group the smaller set surjects
onto the fundamental group of the complement to the bifurcation set.
Hence both fundamental groups map to same braid monodromy
group.

Since a generic vertical line is the image under $p$ of an unfolded 
discriminant curve associated to $g_1$ as in the beginning of the proof,
we have precisely shown what was needed.
\end{proof}

\begin{proof}[Proof of the main Theorem]
Thanks to Prop.\ref{nonBP} it suffices to show
that the discriminant knot group is constant along each
$\mu$-constant stratum which contains a Brieskorn Pham
polynomial $y^3+x^{\nu+1}$.

Let $f$ be any function in this stratum and $\sigma_f$
its $\sigma$-invariant.
Since the analytic equivalence class of $f$ has a representative
of the form
\[
y^3 - P(x)y + x^{\nu+1}\quad \text{with} \quad 
\mbox{$\frac23$}(\nu+1)\leq\ord P,\quad
\deg P \leq \nu-1
\]
we deduce $2\nu \leq 3\sigma_f - 2$ and $\sigma_f \leq \nu$.

Therefore by Prop.\ref{exist} we can unfold the Brieskorn Pham polynomial
by a parameter $s$ such that the $\sigma$-invariant  is $\sigma_f$
for $s\neq0$, and there exists an associated family of
discriminant curves of constant topological type.

Because they are reduced, we can apply the previous proposition to 
see that corresponding unfolded discriminant curves have a complement
with fundamental group isomorphic to the respective discriminant knot
groups.
So by the same argument as in the proof of Prop.\ref{nonBP}
the two discriminant knot groups are isomorphic.

Since $f$ and any deformation of the Brieskorn Pham polynomial with $s$ small
share the same $\sigma$-invariant, we may invoke Cor.\ref{pham_cor}
to have topologically equivalent complements of unfolded
generic discriminant curves.
Thus again the
discriminant knot groups are isomorphic and therefore
constant along the $\mu$-constant stratum of each Brieskorn Pham
polynomial.
\end{proof}

\subsection*{Acknowledgement} 
The author wants to thank Patrick Popescu-Pampu for the invitation to give
a talk at Lille which ultimately caused this project to find its shape.
He gratefully acknowledges the financial support of the ERC 2013 Advanced Research Grant 340258-TADMICAMT.

\bibliography{BibBrieskorn}

\bibliographystyle{alpha}


\end{document}